\documentclass[probth]{svjour1}

\usepackage{amsmath,amsbsy,amsfonts,amssymb,examplep,wasysym,stmaryrd,xcolor,a4wide}

\usepackage[numbers,square,comma]{natbib}
\usepackage{algpseudocode}
\usepackage{algorithm}
\def\eps{\varepsilon}

\def\bW{\boldsymbol{W}}
\def\SSigma{\boldsymbol{\Sigma}}

\def\supp{\mathop{\text{\rm supp}}}

\def\1{\mathbf 1}
\def\RR{\mathbb R}

\def\NN{\mathbb N}
\def\ZZ{\mathbb Z}
\def\ba{\boldsymbol a}

\def\bx{\boldsymbol x}
\def\bt{\boldsymbol t}
\def\bu{\boldsymbol u}

\def\bv{\mathbf v}

\def\llb{\llbracket}
\def\rrb{\rrbracket}

\def\emptyset{\varnothing}

\def\bW{\boldsymbol W}
\def\bB{\boldsymbol B}
\def\bT{\boldsymbol T}

\def\zzeta{\boldsymbol\zeta}

\def\ttheta{\boldsymbol\theta}
\def\ppi{\boldsymbol\pi}
\def\oomega{{\boldsymbol\omega}}

\def\Pb{\mathbf P}
\def\Ex{\mathbf E}
\def\Cov{\textbf{Cov}}

\def\mcA{\mathcal A}
\def\mcB{\mathcal B}

\def\mcF{\mathcal F}

\def\mcJ{\mathcal J}

\def\mcN{\mathcal N}
\def\mcP{\mathcal P}
\def\mcV{\mathcal V}

\def\Pbf{\Pb\!_f}
\def\PbO{\Pb\!_0}
\def\bTA{\bT_{\!\!\mcA}}

\def\bY{{\boldsymbol Y}}
\def\bk{{\boldsymbol k}}
\def\bj{{\boldsymbol j}}

\def\eeta{{\boldsymbol \eta}}
\def\KL{\mathcal K}

\begin{document}

  \title{Statistical inference in compound functional models}

  \titlerunning{Statistical inference in compound functional models}

  \author{Arnak Dalalyan\inst{1} \and \boxed{\rm \text{Yuri Ingster}} \and Alexandre B. Tsybakov\inst{1}}

  \authorrunning{Dalalyan, Ingster and Tsybakov}

\institute{
ENSAE-CREST-GENES\\
3, avenue Pierre Larousse\\
92245 MALAKOFF Cedex, FRANCE\\
  \email{arnak.dalalyan,alexandre.tsybakov@ensae.fr}
\and
St. Petersburg State Electrotechnical University}

  \date{\today}

  \maketitle


\begin{abstract}
We consider a general nonparametric regression model called the compound model. It includes, as special cases, sparse additive regression and nonparametric (or linear) regression with many covariates but possibly a small number of relevant covariates. The compound model is characterized by three main parameters: the structure parameter describing the "macroscopic" form of the compound function,  the "microscopic" sparsity parameter indicating the maximal number of relevant covariates in each component and the usual smoothness parameter corresponding to the complexity of the members of the compound. We find non-asymptotic minimax rate of convergence of estimators in such a model as a function of these three parameters. We also show that this rate can be attained in an adaptive way.
\end{abstract}

\newcommand{\fix}{\marginpar{FIX}}
\def\fix#1{\marginpar{#1}}
\newcommand{\new}{\marginpar{NEW}}
\def\bB{\boldsymbol B}

\section{Introduction}\label{sec:1}

High dimensional statistical inference has known a tremendous development over the past ten years motivated
by applications in various fields such as bioinformatics, computer vision, financial engineering. The
most intensively investigated models in the context of high-dimensionality are the (generalized) linear
models, for which efficient procedures are well known and the theoretical properties are well understood (cf., for instance,
\cite{BRT1,DT08,DT11,Sara_Peter_book}). More recently, increasing interest is demonstrated for studying nonlinear models
in high-dimensional setting \cite{Koltchinskii_Yuan10,Gayraud_Ingster11,ComDal11,Raskutti_et_al11,Suzuki} under various types
of sparsity assumption. The present paper introduces a general framework that unifies these
studies and describes the theoretical limits of statistical procedures in high-dimensional
non-linear problems.

In order to reduce the technicalities and focus on the main ideas, we consider the Gaussian white noise model, which is known to be asymptotically
equivalent, under some natural conditions, to the model of regression \citep{BrownLow96,Reiss08}, as well as to other nonparametric models \citep{DalReiss06,GNZ}.
Thus, we assume that we observe a real-valued Gaussian process $\bY=\{Y(\phi):\phi\in L^2([0,1]^d)\}$
such that
$$
\Ex_{f}[Y(\phi)]=\int_{[0,1]^d} {f}(\bx)\,\phi(\bx)\,d\bx,\qquad
\Cov_f(Y(\phi),Y(\phi'))=\eps^2\int_{[0,1]^d} \phi(\bx)\phi'(\bx)\,d\bx,
$$
for all $\phi,\phi' \in L^2([0,1]^d)$, where $f$ is an unknown function in $L^2([0,1]^d)$, $\Ex_{f}$ and $\Cov_{f}$
are the expectation and  covariance signs, and $\eps$ is some positive number.  It is well known that these two properties
uniquely characterize the probability distribution of a Gaussian process that we will further denote by $\Pbf$ (respectively,
by $\PbO$ if $f\equiv0$). Alternatively, $\bY$ can be considered as a trajectory of the process
$$
dY(\bx)={f}(\bx)\,d\bx+\eps dW(\bx),\qquad\bx\in[0,1]^d,
$$
where $W(\bx)$ is a $d$-parameter Brownian sheet. The parameter $\eps$ is assumed known; in the model of regression it corresponds to the quantity
$\sigma^2 n^{-1/2}$, where $\sigma^2$ is the variance of noise. Without loss of generality, we assume in what follows that $0<\eps<1$.

\subsection{Notation}
First, we introduce some notation. Vectors in finite-dimensional spaces and infinite sequences will be denoted by boldface letters, vector
norms will be denoted by $|\cdot|$ while function norms will be denoted by $\|\cdot\|$. Thus, for $\bv=(v_1,\dots,v_d)\in\RR^d$ we set
$$
|\bv|_0=\sum_{j=1}^d\nolimits \1(v_j\neq 0),\quad |\bv|_\infty=\max_{j= 1,\dots,d} |v_j|,
\quad |\bv|_q^q=\sum_{j=1}^d\nolimits |v_j|^q,\ 1\le q<\infty,
$$
whereas for a function $f:[0,1]^d\to\RR$ we set
$$
\|f\|_\infty=\sup_{\bx\in[0,1]^d} |f(\bx)|,\qquad \|f\|_q^q=\int_{[0,1]^d} |f(\bx)|^q\,d\bx,\ 1\le q<\infty .
$$
We denote by $L^2_0([0,1]^d)$ the subspace of $L^2([0,1]^d)$ containing all the functions $f$ such that
$\int_{[0,1]^d} f(\bx)\,d\bx=0$. The notation $\langle\cdot,\cdot\rangle$ will be used for the inner product in
$L^2([0,1]^d)$, that is $\langle h,\tilde h\rangle=\int_{[0,1]^d} h(\bx)\tilde h(\bx)\,d\bx$ for any
$h,\tilde h\in L^2([0,1]^d)$. For two integers $a$ and $a'$, we denote by $\llb a,a'\rrb$ the set of all integers
belonging to the interval $[a,a']$. We denote by $[t]$ the integer part of  a real number $t$. For a finite set $V$,
we denote by $|V|$ its cardinality. For a vector $\bx\in\RR^d$ and a set of indices $V\subseteq \{ 1,\dots,d\}$, the
vector $\bx_V\in\RR^{|V|}$ is defined as the restriction of $\bx$ to the coordinates with indices belonging to $V$.
For every $s\in\{ 1,\dots, d\}$ and $m\in\NN$, we define $\mcV_s^d=\big\{V\subseteq \{ 1,\dots,d\}:|V|\le s\big\}$
and the set of binary vectors $\mcB_{s,m}^d=\big\{\eeta\in\{0,1\}^{\mcV_s^d}:|\eeta|_0=m\big\}$.  We also use the
notation $M_{d,s}\triangleq|\mcV_s^d|$. We extend these definitions to $s=0$ by setting $\mcV_0^d=\{\emptyset\}$,
$M_{d,0}=1$, $|\mcB_{0,1}^d|=1$, and $|\mcB_{0,m}^d|=0$ for $m>1$. For a vector $\ba$, we denote by $\supp(\ba)$
the set of indices of its non-zero coordinates. In particular, the support $\supp(\eeta)$ of a binary vector
$\eeta=\{\eta_V\}_{V\in \mcV_s^d} \in \mcB_{s,m}^d$ is the set of $V$'s such that $\eta_V=1$.

\subsection{Compound functional model}
In this paper we impose the following assumption on the unknown function $f$.
\begin{description}
\item[\bf Compound functional model:]\it There exists an integer $s\in\{ 1,\dots,d\}$, a binary sequence
$\eeta\in\mcB_{s,m}^d$, a set of functions $\{f_V\in L^2_0([0,1]^{|V|})\}_{V\in \mcV_s^d}$ and a constant $\bar f$ such that
\begin{equation}\label{HRM}
f(\bx)=\bar f+\sum_{V\in \mcV_s^d} f_V(\bx_V)\eta_V =\bar f+ \sum_{V\in\supp(\eeta)} f_V(\bx_V),\qquad\forall \bx\in\RR^d.
\end{equation}
The functions $f_V$ are called the atoms of the compound model.
\end{description}

Note that, under the compound model, $\bar f=\int_{[0,1]^d} f(\bx)\,d\bx$.

The atoms
$f_V$ are assumed to be sufficiently regular, namely, each $f_V$ is an element of a suitable functional class $\Sigma_V$. In particular,
one can consider a smoothness class $\Sigma_V$ and more specifically the Sobolev ball of functions of $s$
variables\footnote{Note that every function of less than $s$ variables can also be considered as a function of $s$ variables.}. In what follows, we will mainly deal with this example.

Given a collection $\SSigma=\{\Sigma_V\}_{V\in\mcV_s^d}$ of
subsets of $L^2_0([0,1]^s)$ and a subset $\tilde\mcB$ of $\mcB_{s,m}^d$, we define the classes
$$
\mcF_{s,m}(\SSigma)=\bigcup_{\eeta\in\tilde \mcB}\nolimits\mcF_{\eeta}(\SSigma),
$$
where
$$
\mcF_{\eeta}(\SSigma)=\Big\{f:\RR^d\to\RR : \exists \bar f\in\RR, \{f_V\}_{V\in\supp(\eeta)}, f_V\in\Sigma_V, \text{ such that }
f=\bar f+\sum_{V\in\supp(\eeta)} f_V \Big\}.
$$
The class $\mcF_{s,m}(\SSigma)$ is defined for any $s\in\{ 0,\dots,d\}$ and any $m \in\{0,\dots, M_{d,s}\}$.
In what follows, we assume that $\tilde\mcB$ is fixed and for this reason we do not include it in the notation.
Examples of $\tilde\mcB$ can be
the set of all $\eeta\in\mcB_{s,m}^d$ such that $V\in\supp(\eeta)$ are pairwise disjoint or of all $\eeta\in\mcB_{s,m}^d$ such that
every set $V$ from $\supp(\eeta)$ has a non-empty intersection with at most one other
set from $\supp(\eeta)$.

It is clear from the definition that the parameters $\big(\eeta,\{f_V\}_{V\in\supp(\eeta)}\big)$
are not identifiable. Indeed, two different collections $\big(\eeta,\{f_V\}_{V\in\supp(\eeta)}\big)$ and $\big(\bar\eeta,\{\bar f_V\}_{V\in\supp(\bar\eeta)}\big)$
may lead to the same compound function $f$. Of course, this is not necessarily an
issue as long as only the problem of estimating $f$ is considered.

We now define the Sobolev classes of functions of many variables that will play the role of $\Sigma_V$. Consider an orthonormal system of functions $\{\varphi_\bj\}_{\bj\in\ZZ^d}$
in $ L^2([0,1]^d)$ such that $\varphi_{\boldsymbol 0}(\bx)\equiv1$. We assume that the system $\{\varphi_\bj\}$ and
the set $\tilde \mcB$ are such that
\begin{equation}\label{3a}
\Big\|\sum_{V\in\supp(\eeta)} \sum_{\substack{\bj: \bj\not=\boldsymbol{0}\\ \supp(\bj)\subseteq V}} \theta_{\bj,V} \varphi_\bj \Big\|^2_2 \le C_*\sum_{V\in\supp(\eeta)}
\sum_{\substack{\bj: \bj\not=\boldsymbol{0}\\ \supp(\bj)\subseteq V}} \theta_{\bj,V}^2,
\end{equation}
for all $\eeta\in\tilde \mcB$ and all square-summable arrays $(\theta_{\bj,V}, \,(\bj,V)\in \ZZ^d\times\mcV_s^d)$,
where $C_*>0$ is a constant independent of $s,m$ and $d$. For example, this condition holds with  $C_*=1$ if $\tilde\mcB$ is
the set of all $\eeta\in\mcB_{s,m}^d$ such that $V\in\supp(\eeta)$ are pairwise disjoint and with $C^*=3/2$ if $\tilde\mcB$ is the set of all $\eeta\in\mcB_{s,m}^d$ such that
every set $V$ from $\supp(\eeta)$ has a non-empty intersection with at most one other
set from $\supp(\eeta)$.

One example of $\{\varphi_\bj\}_{\bj\in\ZZ^d}$ is a tensor product orthonormal basis:
\begin{equation}\label{trig}
\varphi_\bj(\bx)=\bigotimes_{\ell=1}^d\nolimits \varphi_{j_\ell}(x_\ell),\qquad
\end{equation}
where $\bj=(j_1,\dots,j_d)\in\ZZ^d$ is a multi-index and $\{\varphi_{k}\}, \,k\in \ZZ$, is
an orthonormal basis in $ L^2([0,1])$. Specifically, we can take the trigonometric basis
with $\varphi_{0}(u)\equiv 1$ on $[0,1]$, $\varphi_{k}(u)=\sqrt{2}\cos(2\pi\,k u)$ for $k>0$
and $\varphi_{k}(u)=\sqrt{2}\sin(2\pi\,k u)$ for $k<0$. To ease notation, we set
$\theta_\bj[{f}]=\langle {f},\varphi_\bj\rangle$ for $\bj\in\ZZ^d$.

For any set of indices  $V\subseteq\{1,\dots,d\}$ and any $\beta>0$, $L>0$, we define the
Sobolev class of functions
\begin{align}\label{Sobolev}
W_V(\beta,L)=\bigg\{g\in L^2_0([0,1]^{d}) :\quad g=\sum_{\bj\in\ZZ^d:\supp(\bj)\subseteq V}
\theta_\bj[g]\varphi_\bj\quad \text{and} \quad \sum_{\bj\in\ZZ^d} |\bj|_\infty^{2\beta}
\theta_\bj[g]^2\le L\bigg\}.
\end{align}
Assuming that $\{\varphi_\bj\}$ is the trigonometric basis and $f$ is periodic with period
one in each coordinate, \textit{i.e.}, $f(\bx+\bj)=f(\bx)$ for every $\bx\in\RR^d$ and every
$\bj\in\ZZ^d$, the condition $f_V\in W_{V}(\beta,L)$ can be interpreted as the square
integrability of all partial derivatives of $f_V$ up to the order $\beta$.

Let us give some examples of compound models.
\begin{itemize}
\item {\it Additive models} are the special case $s=1$ of compound models. Here, additive
models are understood in a wider sense than originally defined by  \citet{Stone85}. Namely,
for $s=1$ we have the model
$$
f(\bx)=\bar f+\sum_{j\in J} f_j(x_j), \qquad \bx=(x_1,\dots,x_d)\in\RR^d,
$$
where $J$ is any (unknown) subset of indices and not necessarily $J=\{1,\dots, d\}$. Estimation
and testing problems in this model when the atoms belong to some smoothness classes  have been
studied in \citet{Ingster_Lepski03,Meier_vdG_Buhlmann,Koltchinskii_Yuan10,Raskutti_et_al11,Gayraud_Ingster11,Suzuki}.
\item {\it Single atom models} are the special case $m=1$ of compound models. If $m=1$
    we have $f(\bx)=f_V(\bx_V)$ for some unknown $V\subseteq \{ 1,\dots,d\}$, \textit{i.e.},
    there exists only one set $V$ for which $\eta_V=1$, and $|V|\le s$. Estimation and
    variable selection in this model were considered by \citet{BertinLecue,CD11a,mosci_rosasco}.
    The case of small $s$ and large $d$ is particularly interesting in the context of sparsity.
    In a parametric model, when $f_V$ is a linear function, we are back to the sparse
    high-dimensional linear regression setting, which has been extensively studied, see, e.g.,
    \citet{Sara_Peter_book}.
\item {\it Tensor product models.} Let $\mcA$ be a given finite subset of $\ZZ$, and assume
    that $\varphi_\bj$ is a tensor product basis defined by~(\ref{trig}). Consider the following
    parametric class of functions
    \begin{align}\label{tensor}
    \bT_{\eeta}(\mcA)=\Big\{f:\RR^d\to\RR : \exists \bar f ,  \{\theta_{\bj,V}\}, \text{ such that }
    f=\bar f+\sum_{V\in\supp(\eeta)}\sum_{\bj\in\mcJ_{V,\mcA}} \theta_{\bj,V}\varphi_\bj \Big\},
    \end{align}
    where
    \begin{align}
    \mcJ_{V,\mcA}=\Big\{ \bj\in\mcA^d : \supp(\bj) \subseteq V \Big\}.
    \end{align}
    We say that function $f$ satisfies the tensor product model if it belongs to the set $\bT_\eeta(\mcA)$ for some $\eeta\in \tilde\mcB$.
    We define 
    $$
    \mcF_{s,m}(\bTA)=\bigcup_{\eeta\in\tilde \mcB}\nolimits \bT_\eeta(\mcA).
    $$
    Important examples are sparse high-dimensional multilinear/polynomial systems. Motivated respectively
    by applications in genetics and signal processing, they have been recently studied by \citet{Bobak_Novak}
    in the context of compressed sensing without noise and by \citet{Kekatos11} in the case where the
    observations are corrupted by a Gaussian noise. With our notation, the models they considered are the
    tensor product models with $\mcA=\{0,1\}$ (linear basis functions $\varphi_j$) in the multilinear model
    of \cite{Bobak_Novak} and $\mcA=\{-1,0,1\}$ in the Volterra filtering problem of \cite{Kekatos11}
    (second-order Volterra systems with $\varphi_0(x)\equiv1$, $\varphi_1(x)\propto (x-1/2)$ and
    $\varphi_{-1}(x)\propto x^2-x+1/6$). More generally, the set $\mcA$ should be of small cardinality to
    guarantee efficient dimension reduction.  Another approach is to introduce hierarchical structures on
    the coefficients of tensor product representation \cite{BRT2,Bach09}.
\end{itemize}

In what follows, we assume that $f$ belongs to the functional class $\mcF_{s,m}(\SSigma)$
where either $\SSigma=\{W_{V}(\beta,L)\}_{V\in\mcV_s^d}\triangleq\bW(\beta,L)$ or $\SSigma=\bTA$.

The compound model is described by three main parameters, which are the dimension $m$ that we call the {\it macroscopic} parameter and that characterizes the complexity of possible structure vectors $\eeta$,  the dimension $s$ of atoms in the compound that we call the {\it microscopic} parameter, and the complexity of functional class $\SSigma$. The latter can be described by entropy numbers of $\SSigma$ in convenient norms, and in the particular case of Sobolev classes, it is naturally characterized by the smoothness parameter $\beta$. The integers $m$ and $s$ are ``effective dimension" parameters. As soon as they grow, the structure becomes less pronounced and the compound model approaches the global nonparametric regression in dimension $d$, which is known to suffer from the curse of dimensionality already for moderate $d$. Therefore, an interesting case is the sparsity scenario where $s$ and/or $m$ are small.

\section{Overview of the results and relation to the previous work }
Several statistical problems arise naturally in the context of compound functional model.
\begin{description}
\item[\bf Estimation of $f$.] This is the subject of the present paper. We measure the risk of arbitrary estimator $\widetilde f_\eps$  by its mean integrated squared error $\Ex_f[\|\widetilde f_\eps-f\|_2^2]$ and we study the minimax risk
    $$
    \inf_{\widetilde f_\eps}\sup_{f\in \mcF_{s,m}(\SSigma)} \Ex_f[\| \widetilde f_\eps-f\|_2^2],
    $$
    where $\inf_{\widetilde f_\eps}$ denotes the minimum over all estimators\footnote{We focus our attention on the behavior of the expected error of estimation. Alternatively,
    one can be interested in establishing similar type of upper bounds on the error of estimation that hold true with large probability \cite{Dai_et_al,Lecue_Mend1}.}.
    A first general question is to establish the minimax rates
    of estimation, \textit{i.e.}, to find values $\psi_{s,m,\eps}(\SSigma)$ such that
    $$
    \inf_{\widetilde f_\eps}\sup_{f\in \mcF_{s,m}(\SSigma)} \Ex_f[\|\widetilde f_\eps-f\|_2^2] \asymp \psi_{s,m,\eps}(\SSigma),
    $$
    when $\SSigma$ is a Sobolev,
    H\"older or other class of functions.
    A second question is to construct optimal estimators in a minimax sense, \textit{i.e.}, estimators $\widehat f_\eps$ such that
    \begin{equation}\label{10}
 \sup_{f\in \mcF_{s,m}(\SSigma)} \Ex_f[\|\widehat f_\eps-f\|_2^2] \le C \psi_{s,m,\eps}(\SSigma),
   \end{equation}
 for some constant $C$ independent of  $s,m,\eps$ and $\SSigma$.
    Some results on minimax rates of estimation of $f$ are available only for the case $s=1$ (cf. the discussion below).
    Finally, a third question that we address here is whether the optimal rate can be attained adaptively, \textit{i.e.},
    whether one can construct an estimator $\widehat f_\eps$  that satisfies (\ref{10}) simultaneously for all $s,m,\beta$
    and $L$ when $\SSigma=\bW(\beta,L)$. We will show that the answer to this question is positive.
    \medskip
\item[\bf Variable selection.] Assume that $m=1$. This means that $f(\bx)=f_V(\bx_V)$ for some unknown
    $V\subseteq \{ 1,\dots,d\}$, \textit{i.e.}, there exists only one set $V$ for which $\eta_V=1$ (a single atom model). Then
    it is of interest to identify $V$ under the constraint $|V|\le s$. In particular, $d$
    can be very large while $s$ can be small. This corresponds to estimating the relevant covariates
    and generalizes the problem of selection of sparsity pattern in linear regression. An estimator
    $\widehat V_n\subseteq \{ 1,\dots,d\}$ of $V$ is considered as good, if the probability $\Pb(\widehat V_n=V)$ is
    close to one. \medskip
\item[\bf Hypotheses testing (detection):] The problem is to test the hypothesis $H_0: f\equiv 0$
    (no signal) against the alternative $H_1 : f\in \mcA$, where $\mcA=\big\{f\in\mcF_{s,m}(\SSigma): \|f\|_2\ge r\big\}$.
    Here, it is interesting to characterize the minimax rates of separation $r>0$ in terms of $s$, $m$ and $\SSigma$.
\end{description}

Some of the above three problems have been studied
in the literature for special cases $s=1$ (additive model) and $m=1$ (single atom model). \citet{Ingster_Lepski03} studied the problem of testing in additive model
and provided asymptotic minimax rates of separation. Sharp asymptotic optimality under additional assumptions in the same problem
was obtained by \citet{Gayraud_Ingster11}. Recently, \citet{CD11a} established tight conditions for
variable selection in the single atom model. We also mention an earlier work of \citet{BertinLecue} dealing with variable selection.

The problem of estimation has been also considered for additive model and class $\SSigma$ defined as a reproducing kernel Hilbert space, cf. \citet{Koltchinskii_Yuan10,Raskutti_et_al11}. In particular, these papers showed that if $s=1$ and
$\SSigma=\bW(\beta,L)$ is a Sobolev class, then there is an estimator of $f$ for which
the mean integrated squared error converges to zero at the rate
\begin{align} \label{rate1}
\max\Big(m\eps^{4\beta/(2\beta+1)},\;m\eps^2\log d\Big).
\end{align}
Furthermore, \citet[Thm.\ 2]{Raskutti_et_al11} provided the following  lower bound on the minimax risk:
\begin{align} \label{rate2}
\max\Big(m\eps^{4\beta/(2\beta+1)},\;m\eps^2\log \Big(\frac{d}{m}\Big)\Big).
\end{align}
Note that when $m$ is proportional to $d$, this lower bound departs from the upper
bound in a logarithmic way. It should also be noted that the upper bounds in these papers are achieved by
estimators that are not adaptive in the sense that they require the knowledge of the smoothness index $\beta$.

In this paper, we establish non-asymptotic upper and lower bounds on the minimax risk for the  model
with Sobolev smoothness class $\SSigma=\bW(\beta,L)$. We will prove that, up to a multiplicative constant, the minimax risk behaves itself as
\begin{align} \label{rate3}
\max\bigg\{mL^{s/(2\beta+s)}\eps^{4\beta/(2\beta+s)},\;ms\eps^2\log \bigg(\frac{d}{sm^{1/s}}\bigg)\bigg\}\wedge  L
\end{align}
(we assume here $d/(sm^{1/s})>1$, otherwise a constant factor greater than 1 should be inserted under the logarithm, cf. the results below).
In addition, we demonstrate that this rate can be reached in an adaptive way that is without the knowledge
of $\beta$, $s$, and $m$. The rate  (\ref{rate3}) is non-asymptotic, which explains, in particular, the presence of minimum with constant $L$ in (\ref{rate3}).
For $s=1$, \textit{i.e.}, for the additive regression model, our rate matches the lower bound of \cite{Raskutti_et_al11}.

For $m=1$, \textit{i.e.}, when $f(\bx)=f_V(\bx_V)$ for some unknown $V\subseteq \{ 1,\dots,d\}$ (the single atom model), the minimax rate
of convergence takes the form
\begin{align} \label{rate4}
\max\bigg\{L^{s/(2\beta+s)}\eps^{4\beta/(2\beta+s)},\;s\eps^2\log \bigg(\frac{d}{s}\bigg)\bigg\}\wedge  L.
\end{align}
This rate accounts for two effects, namely, the accuracy of nonparametric estimation of $f$ for fixed macroscopic structure parameter $\eeta$, cf. the first term $\sim \eps^{4\beta/(2\beta+s)}$, and the complexity of the structure itself (irrespective to the nonparametric nature of microscopic components $f_V(\bx_V)$).  In particular, the second term $\sim s\eps^2\log ({d}/{s})$ in (\ref{rate4}) coincides with the optimal rate of prediction in linear regression model under the standard sparsity assumption. This is what we obtain in the limiting case when $\beta$ tends to infinity. It is important to note that the optimal rates depend only logarithmically on the ambient dimension $d$.  Thus, even if $d$ is large, the rate optimal estimators achieve nice performance under the sparsity scenario when $s$ and $m$ are small.

\section{The estimator and upper bounds on the minimax risk}\label{sec:upper}

In this section, we suggest an estimator attaining the
minimax rate. It is constructed in the following two steps.
\begin{description}
\item[\bf Constructing weak estimators.] At this step, we proceed as if the macroscopic structure parameter $\eeta$
was known and denote by $V_1,\ldots,V_m$ the elements of the support of $\eeta$. The goal is to provide for each $\eeta$ a family of
``simple'' estimators of $f$---indexed by some parameter $\bt$---containing a rate-minimax one.  To this end, we first project $\bY$ onto
the basis functions $\{\varphi_\bj:{|\bj|}_\infty\le \eps^{-2}\}$ and denote
\begin{equation}\label{GSM}
\bY_\eps=(Y_\bj\triangleq Y(\varphi_\bj): \bj\in\ZZ^d, \, {|\bj|}_\infty\le \eps^{-2}).
\end{equation}
Then, we consider a collection $\{\widehat\ttheta_{\bt,\eeta}:\ \bt\in \ZZ^m\cap [1,\eps^{-2}]^m\}$ of projection estimators of
the vector $\ttheta_\eps=(\theta_\bj[f])_{\bj\in\ZZ^d: |\bj|_\infty\le \eps^{-2}}$. The role of each component $t_\ell$ of $\bt$ is
to indicate the cut-off level of the coefficients $\theta_{\bj}$ corresponding to the atom $f_{V_\ell}$, that is the level of indices
beyond of which the coefficients are estimated by $0$.

To be more precise, for an integer-valued vector $\bt=(t_{V_\ell}, \ell=1,\dots, m)\in[0, \eps^{-2}]^m$ 
we set
$\widehat\ttheta_{\bt,\eeta}=(\widehat\theta_{\bt,\eeta,\bj}:\bj\in\ZZ^d, \,{|\bj|}_\infty\le \eps^{-2})$,
where $\widehat\theta_{\bt,\eeta,\boldsymbol0}=Y_{\boldsymbol0}$ and
$$
\widehat\theta_{\bt,\eeta,\bj}=
\begin{cases}
Y_\bj, & \exists \ell \text{ s.\,t. }\supp(\bj)\subseteq V_\ell\, , \, {|{\bj}|}_\infty\in[1, t_{V_\ell}],  \\
0, & \text{otherwise}
\end{cases}
$$
if $\bj\ne\boldsymbol0$. Based on these estimators of the coefficients of $f$, we recover the function $f$ using the estimator
$$
\widehat f_{\bt,\eeta}(\bx)=\sum_{\bj\in\ZZ^d:{|\bj|}_\infty\le \eps^{-2}}\nolimits \widehat\theta_{\bt,\eeta,\bj} \varphi_\bj(\bx).
$$
\end{description}


\begin{description}
\item[\bf Smoothness- and structure-adaptive estimation:]
The goal in this step is to combine the weak estimators $\{\widehat f_{\bt,\eeta}\}_{\bt,\eeta}$ in order
to get a structure and smoothness adaptive estimator of $f$ with a risk which is as small as possible.
To this end, we use a version of exponentially weighted aggregate \citep{Leung_Barron06,DT08,DT11} in the spirit
of sparsity pattern aggregation as described in~\cite{RigTsy11,RigTsy12}. More precisely, for every
pair of integers $(s,m)$ such that $s\in\{1,\ldots,d\}$ and $m\in\{1,\ldots,M_{d,s}\}$, we define prior probabilities for
$(\bt,\eeta)\in\llb0,\eps^{-2}\rrb^m\times (\mcB_{s,m}^d\setminus\mcB_{s-1,m}^d)$  by
\begin{align}\label{prior}
\pi_{\bt,\eeta}=\frac{2^{-sm}}{H_d(1+[\eps^{-2}])^m|\mcB_{s,m}^d\setminus\mcB_{s-1,m}^d|},\qquad
H_d=\sum_{s=0}^d\sum_{m=1}^{M_{d,s}} 2^{-sm}\le  e.
\end{align}
For $s=0$ and the unique $\eeta_0\in\mcB_{0,1}^d$ we consider only one weak estimator $\widehat\ttheta_{\bt,\eeta_0}$ with all entries
zero except for the entry $\widehat\theta_{\bt,\eeta_0,\boldsymbol 0}$,  which is equal to $Y_{\boldsymbol0}$. We set $\pi_{\bt,\eeta_0}=1/H_d$.
It is easy to see that $\ppi=\Big(\pi_{\bt,\eeta};(\bt,\eeta)\in \bigcup_{s,m} \{\llb0,\eps^{-2}\rrb^m\times \mcB_{s,m}^d\}\Big)$
defines a probability distribution.
For any pair $(\bt,\eeta)$ we
introduce the penalty function $$\text{pen}(\bt,\eeta)=2\eps^2\prod_{V\in\supp(\eeta)} (2t_V+1)^{|V|}$$ and  define the vector  of  coefficients
$\widehat \ttheta_\eps=(\widehat \theta_{\eps,\bj}:\bj\in\ZZ^d, \, {|\bj|}_\infty\le \eps^{-2})$ by
\begin{equation}
\widehat \ttheta_\eps=\sum_{s=1}^d\sum_{m=1}^{M_{d,s}}\sum_{(\bt,\eeta)} \widehat\ttheta_{\bt,\eeta}
\frac{\exp\big\{-\frac1{4\eps^2}\big(|\bY_\eps-\widehat\ttheta_{\bt,\eeta}|_2^2+\text{pen}(\bt,\eeta)\big)\big\}\pi_{\bt,\eeta}}
{\sum_{\bar s=1}^d\sum_{\bar m=1}^{M_{d,\bar s}}\sum_{(\bar\bt,\bar\eeta)}\exp\big\{-\frac1{4\eps^2}
\big(|\bY_\eps-\widehat\ttheta_{\bar\bt,\bar\eeta}|_2^2+\text{pen}(\bar\bt,\bar\eeta)\big)\big\}\pi_{\bar\bt,\bar\eeta}},
\end{equation}
where the summations $\sum_{(\bt,\eeta)}$ and $\sum_{(\bar\bt,\bar\eeta)}$ correspond to $(\bt,\eeta)
\in \llb0,\eps^{-2}\rrb^m\times(\mcB^d_{s,m}\setminus \mcB^d_{s-1,m})$ and $(\bar\bt,\bar\eeta)\in \llb0,\eps^{-2}\rrb^{\bar m}\times
(\mcB^d_{\bar s,\bar m}\setminus\mcB^d_{\bar s-1,\bar m})$, respectively. The final estimator of $f$ is
$$
\widehat f_\eps(\bx)=\sum_{\bj\in\ZZ^d:{|\bj|}_\infty\le \eps^{-2}}\nolimits \widehat\theta_{\eps,\bj}\varphi_\bj(\bx),
\qquad\forall\bx\in [0,1]^d.
$$
\end{description}
Note that each $\widehat\ttheta_{\bt,\eeta}$ is a projection estimator of the vector $\ttheta=(\theta_\bj[f])_{\bj\in\ZZ^d}$. Hence,
$\widehat f_\eps$ is a convex combination of projection estimators. We also note that, to construct $\widehat f_\eps$, we only need to know
$\eps$ and $d$. Therefore, the estimator is adaptive to all other parameters of the model, such as $s$, $m$, the parameters that define the class $\SSigma$ and the choice of a particular subset $\tilde\mcB$ of $\mcB^d_{s,m}$.

The following theorem gives an upper bound on the risk of the estimator $\widehat f_\eps$ when $\SSigma=\bW(\beta,L)$.

\begin{theorem}\label{thm:1}
Let $\beta>0$ and $L>0$ be such that $\log(\eps^{-2})\ge (2\beta)^{-1}\log(L)$, $L>\eps^{2}\log(e\eps^{-2})^{\frac{2\beta+s}{s}}$. Let $\tilde\mcB$ be any subset of $\mcB^d_{s,m}$. Assume that condition (\ref{3a}) holds. Then, for some constant $C({\beta})>0$ depending only on $\beta$ we have
\begin{align}\label{upper}
\sup_{f\in \mcF_{s,m}(\bW(\beta,L))}\Ex_f[\|\widehat f_\eps-f\|_2^2]\le (6L)\wedge \Big(m\Big\{ C({\beta})L^{\frac{s}{2\beta+s}}
\eps^{\frac{4\beta}{2\beta+s}}+4s\eps^2\log\Big(\frac{2e^3d}{sm^{1/s}}\Big)\Big\}\Big)\,.
\end{align}
\end{theorem}

\begin{proof}
Since the functions $\varphi_\bj$ are orthonormal, $\bY_\eps$ is composed of independent Gaussian
random variables with common variance equal to $\eps^2$. Thus, the array $\bY_\eps$ defined by (\ref{GSM})
obeys the Gaussian sequence model studied in \cite{Leung_Barron06}. Therefore, using Parseval's theorem and
\cite[Cor. 6]{Leung_Barron06} we obtain that the estimator $\widehat f_\eps$ satisfies, for all $f$,
\begin{align}\label{oracle1}
\Ex_f[\|\widehat f_\eps-f\|_2^2]
		&\le \min_{\bt,\eeta}\Big(\Ex_f[\|\widehat f_{\bt,\eeta}-f\|_2^2]+4\eps^2\log(\pi_{\bt,\eeta}^{-1})\Big),
\end{align}
where the minimum is taken over all $(\bt, \eeta)\in \bigcup_{s,m} \{\llb0,\eps^{-2}\rrb^m\times \mcB_{s,m}^d\}$. Denote by $\eeta_0$ the unique element
of $\mcB_{0,1}^d$ for which $\supp(\eeta)=\{\emptyset\}$. The corresponding estimator $\widehat f_{\bt,\eeta_0}$ coincides with
the constant function equal to $Y_{\boldsymbol 0}$ and its risk is bounded by $\eps^2+L$ for all $f\in  \mcF_{s,m}(\bW(\beta,L))$. Therefore,
\begin{align}\label{11}
\sup_{f\in \mcF_{s,m}(\bW(\beta,L))}\Ex_f[\|\widehat f_\eps-f\|_2^2]
		&\le \sup_{f\in \mcF_{s,m}(\bW(\beta,L))}\Ex_f[\|\widehat f_{\bt,\eeta_0}-f\|_2^2]+4\eps^2\log(\pi_{\bt,\eeta_0}^{-1})\nonumber\\
		&\le \eps^2+L+4\eps^2\le 6L.
\end{align}
Take now any $f\in \mcF_{s,m}(\bW(\beta,L))$, and let $\eeta^*\in\tilde\mcB\subseteq\mcB_{s,m}^d$ be such that $f\in \mcF_{\eeta^*}(\bW(\beta,L))$. Then it follows from (\ref{oracle1}) that
\begin{align}\label{oracle2}
\Ex_f[\|\widehat f_\eps-f\|_2^2]
		&\le \min_{\bt\in\llb0,\eps^{-2}\rrb^m}\Big(\Ex_f[\|\widehat f_{\bt,\eeta^*}-f\|_2^2]+4\eps^2\log(\pi_{\bt,\eeta^*}^{-1})\Big)\nonumber\\
		&\le \min_{\bt\in\llb0,\eps^{-2}\rrb^m}\Ex_f[\|\widehat f_{\bt,\eeta^*}-f\|_2^2]+4\eps^2 \big(m\log(2\eps^{-2})+ms\log(2)+\log(e|\mcB_{s,m}^d|)\big).
\end{align}
Note that for all
$d,s\in\NN$ such that $s\le d$ we have
\begin{align}\label{eq:1}
M_{d,s}=\sum_{\ell=0}^s \binom{d}{\ell}\le \bigg(\frac{ed}{s}\bigg)^{s}\quad \mbox{and}
\quad |\mcB_{s,m}^d|\le \binom{M_{d,s}}{m}\le \bigg(\frac{eM_{d,s}}{m}\bigg)^m\le \bigg(\frac{e^2d}{sm^{1/s}}\bigg)^{ms}.
\end{align}
Also, we have the following bound on the risk of estimator $\widehat f_{\bt,\eeta}$ for each $\eeta\in\tilde\mcB$ and for an appropriate choice of
the bandwidth parameter $\bt\in\llb0,\eps^{-2}\rrb^m$.

\begin{lemma}\label{lem:1}
Let $\beta>0$,  $L\ge \eps^{2}$ be such that $\log(\eps^{-2})\ge (2\beta)^{-1}\log(L)$. Let $\bt\in\llb0,\eps^{-2}\rrb^m$
be a vector with integer coordinates
$t_{V_\ell}=[(L/(3^{|V_\ell|}\eps^2))^{1/(2\beta+|V_\ell|)}\wedge \eps^{-2}]$, $\ell=1,\dots,m$.
Assume that condition (\ref{3a}) holds. Then
\begin{align}\label{eq:lem1}
 \sup_{f\in \mcF_{\eeta}(\bW(\beta,L))}\Ex_f[\|\widehat f_{\bt,\eeta}-f\|_2^2]
	&\le 2C_* 3^{2\beta\wedge s}\, m\,L^{s/(2\beta+s)}\eps^{4\beta/(2\beta+s)},\qquad \forall\; \eeta\in\tilde\mcB\subset{\mcB^d_{s,m}}.
\end{align}
\end{lemma}
Proof of this lemma is given in the appendix.

Combining (\ref{oracle2}) with (\ref{eq:1}) and (\ref{eq:lem1}) yields the following
upper bound on the risk of $\widehat f_\eps$ :
$$
\sup_{f\in \mcF_{s,m}(\bW(\beta,L))}\Ex_f[\|\widehat f_\eps-f\|_2^2]\le m\bigg\{ C_{\beta}L^{\frac{s}{2\beta+s}}
\eps^{\frac{4\beta}{2\beta+s}}+4\eps^2\log(2\eps^{-2})
+4s\eps^2\log\bigg(\frac{2e^3d}{sm^{1/s}}\bigg)\bigg\}
$$
where $C_{\beta}>0$ is a constant depending only on $\beta$. The assumptions of the theorem guarantee that $
\eps^2\log(2\eps^{-2})\le L^{\frac{s}{2\beta+s}}
\eps^{\frac{4\beta}{2\beta+s}}$, so that the desired result follows from (\ref{11}) and the last display.

\end{proof}

The behavior of the estimator $\widehat f_\eps$ in the case $\SSigma=\bTA$ is described in the next theorem.
\begin{theorem}\label{thm:2}
Assume that $k=\max\{|\ell|: \,  \ell \in \mcA\}<\eps^{-2}$. Then
\begin{align}\label{upper:tensor}
\sup_{f\in \mcF_{s,m}(\bTA)}\Ex_f[\|\widehat f_\eps-f\|_2^2]\le
 m\eps^2 \Big\{ (2k+1)^s
+4\log(2\eps^{-2})+4s\log\Big(\frac{2e^3d}{sm^{1/s}}\Big)\Big\}\,.
\end{align}
\end{theorem}

Proof of Theorem \ref{thm:2} follows the same lines as that of Theorem~\ref{thm:1}.  We take $f\in \mcF_{s,m}(\bTA)$, and let $\eeta^*\in\tilde\mcB\subseteq\mcB_{s,m}^d$ be such that $f\in \mcF_{\eeta^*}(\bTA)$. Let $\bt^*\in \RR^m$ be the vector with all coordinates equal to $k$. Then the same argument as in (\ref{oracle2}) yields
\begin{align}\label{oracle4}
\Ex_f[\|\widehat f_\eps-f\|_2^2]
				&\le \Ex_f[\|\widehat f_{\bt^*,\eeta^*}-f\|_2^2]+4\eps^2 \big(m\log(2\eps^{-2})+ms\log(2)+\log(e|\mcB_{s,m}^d|)\big).
\end{align}
We can write ${\rm supp}(\eeta^*)=\{V_1,\dots,V_m\}$ where $|V_\ell|\le s$.  Since the model is parametric, there is no bias term in the expression for the risk on the right hand side of  (\ref{oracle4}) and we have (cf. (\ref{eq:4})):
\begin{align*}\label{oracle5}
\Ex_f[\|\widehat f_{\bt^*,\eeta^*}-f\|_2^2]& \le
\sum_{\ell=1}^m\sum_{\bj:\supp(\bj)\subseteq V_\ell}\eps^2\1_{\{{|\bj|}_\infty\le k\}}
		\le m\eps^2(2k+1)^{s}.
\end{align*}
Together with  (\ref{oracle4}), this implies (\ref{upper:tensor}).

The bound of Theorem~\ref{thm:2} is particularly interesting when $k$ and $s$ are small. For the examples of multilinear and polynomial systems \cite{Bobak_Novak,Kekatos11} we have $k=1$. We also note that the result is much better than what can be obtained by using the Lasso. Indeed, consider the simplest case of single atom tensor product model ($m=1$).  Since we do not know $s$, we need to run the Lasso in the dimension $p=k^d$ and we can only guarantee the rate $\eps^2\log p=d \eps^2\log k$, which is linear in the dimension $d$. If $d$ is very large and $s\ll d$, this is much slower than the rate of Theorem~\ref{thm:2}.

\section{Lower bound}

In this section, we prove a minimax lower bound on the risk of any estimator over the class $\mcF_{s,m}(\bW(\beta,L))$. We will assume that $\{\varphi_\bj\}$ is the tensor-product trigonometric basis and $\tilde \mcB= \tilde\mcB_{s,m}^d$ where $\tilde\mcB_{s,m}^d$ denotes the set of all $\eeta\in \mcB_{s,m}^d$ such that the sets $V\in {\rm \supp}(\eeta)$ are disjoint. Then condition (\ref{3a}) holds with equality and $C_*=1$. We will split the proof into two steps. First, we establish a lower bound
on the minimax risk in the case of known structure $\eeta$, \textit{i.e.}, when $f$ belongs
to the class $\mcF_\eeta(\bW(\beta,L))$ for some known parameters $\eeta\in \tilde \mcB$ and  $\beta,L>0$. We will
show that the minimax risk tends to zero with the rate not faster than $m\eps^{4\beta/(2\beta+s)}$. In a second step, we will
prove that if $\eeta$ is unknown, then the minimax rate is bounded from below by $ms\eps^2(1+\log(d/(sm^{1/s})))$ if
the function $f$ belongs to $\mcF_\eeta(\Theta)$ for a set $\Theta$ spanned by the tensor products involving only the functions $\varphi_1$ and $\varphi_{-1}$ of various arguments.

\subsection{Lower bound for known structure $\eeta$}

\begin{proposition}\label{prop1}
Let $\{\varphi_\bj\}$ be the tensor-product trigonometric basis and let $s,m,d$ be positive integers satisfying $d\ge sm$.
Assume that $L\ge \eps^2$. Then there exists an absolute constant $C>0$ such that
$$
\inf_{\widehat f}\sup_{f\in \mcF_{\eeta}(\bW(\beta,L))}\Ex_f[\|\widehat f-f\|_2^2]\ge CmL^{s/(2\beta+s)}\eps^{4\beta/(2\beta+s)},\qquad
\forall\,\eeta\in\tilde \mcB_{s,m}^d.
$$
\end{proposition}

\begin{proof}
Without loss of generality assume that $m = 1$. We will also assume that $L=1$ (this is without loss of generality as well,
since we can replace $\eps$ by $\eps/\sqrt{L}$ and by our assumption this quantity is less than 1). After a renumbering if needed,
we can assume that $\eeta$ is such that $\eeta_V=1$ for $V=\{1,\ldots,s\}$ and  $\eeta_V=0$ for $V\ne\{1,\ldots,s\}$.

Let $t$ be an integer not smaller than $4$. Then, the set $I$ of all multi-indices $\bk\in\ZZ^{s}$ satisfying
$|\bk|_\infty\le t$ is of cardinality $|I|\ge 9$. For any $\oomega=(\omega_k, k\in I) \in\{0,1\}^I$, we set
$f_{\oomega}(\bx)=\gamma\sum_{\bk\in I} \omega_\bk \varphi_{\bk}(x_1,\ldots,x_s)$,
where $\varphi_{\bk}(x_1,\ldots,x_s)=\prod_{j=1}^s \varphi_{k_j}(x_j)$, $\bk=(k_1,\dots,k_s)$,  is an element of the tensor-product trigonometric basis and
$\gamma>0$ is a parameter to be chosen later. In view of the orthonormality of the basis functions $\varphi_\bk$, we have
\begin{equation}\label{eq:7}
\|f_\oomega\|_2^2=\gamma^2|\oomega|_1, \quad \forall \ \oomega\in \{0,1\}^I.
\end{equation}
Therefore, we have $\sum_{\bk}|\bk|_\infty^{2\beta}\theta_\bk[f_\oomega]^2\le t^{2\beta} \|f_\oomega\|_2^2\le t^{2\beta}\gamma^2(2t+1)^s\le \gamma^2(2t+1)^{2\beta+s}$. Thus,
the condition $\gamma^2(2t+1)^{2\beta+s}\le 1$ ensures that all the functions $f_\oomega$ belong to $W(\beta,1)$.

Furthermore, for two vectors $\oomega, \oomega'\in \{0,1\}^I$ we have
$\|f_\oomega-f_{\oomega'}\|_2^2=\gamma^2|\oomega-\oomega'|_1$.
Note that the entries of the vectors $\oomega,\oomega'$ are either 0 or 1, therefore the $\ell_1$ distance
between these vectors coincides with the Hamming distance. According to the Varshamov-Gilbert lemma \cite[Lemma 2.9]{Tsybakov09}, there exists a set $\Omega\subset\{0,1\}^I$ of cardinality at least $2^{|I|/8}$
such that it contains the zero element and the pairwise distances
$|\oomega-\oomega'|_1$ are at least $|I|/8$ for any pair $\oomega,\oomega'\in \Omega$.

We can now apply Theorem~2.7 from \cite{Tsybakov09} that asserts that if, for some $\tau>0$, we have $\min_{\oomega,\oomega'\in\Omega} \|f_\oomega-f_{\oomega'}\|_2\ge 2\tau>0$, and
\begin{equation}\label{eq:9}
\frac1{|\Omega|}\sum_{\oomega\in\Omega}\nolimits \KL(\Pb_{f_\oomega},\PbO)\le \frac{\log |\Omega|}{16},
\end{equation}
where  $\KL(\cdot,\cdot)$ denotes the Kullback-Leibler divergence, then $\inf_{\widehat f}\max_{\oomega\in \Omega}\Ex_{f_\oomega}[\|\widehat f-f_{\oomega}\|_2^2]\ge c'\tau^2$ for some absolute constant $c'>0$. In our case, we set $\tau=\gamma\sqrt{ |I|/32}$.
Combining (\ref{eq:7}) and the fact that the Kullback-Leibler divergence between the Gaussian measures $\Pb_{f}$ and $\Pb_g$ is given by
$\frac12\eps^{-2}\|f-g\|_2^2$, we obtain
$\frac1{|\Omega|}\sum_{\oomega\in\Omega} \KL(\Pb_{f_\oomega},\PbO)\le\frac{1}{2}{\eps^{-2}\gamma^2 |I|}\,$.
If $\gamma^2\le (\log 2)\eps^2/64$, then (\ref{eq:9}) is satisfied and $\tau^2 =\gamma^2(2t+1)^s/32$ is a lower bound on the rate of convergence of
the minimax risk.

To finish the proof, it suffices to choose $t\in \NN$ and $\gamma>0$ satisfying the following three conditions: $t\ge 4$, $\gamma^2\le (2t+1)^{-2\beta-s}$  and
$\gamma^2\le \eps^2\log(2)/64$. For the choice $\gamma^{-2}=(2t+1)^{2\beta+s}+\eps^{-2}64/\log(2)$ and$t=[4\eps^{-2/(2\beta+s)}]$ all these conditions are
satisfied and $\tau^2 \ge c_1\eps^{4\beta/(2\beta+s)}$ for
some absolute positive constant $c_1$.
\end{proof}

\subsection{Lower bound for unknown structure $\eeta$}
\begin{proposition}
Let the assumptions of Proposition \ref{prop1} be satisfied. Then there exists an absolute constant $C'>0$ such that
$$
\inf_{\widehat f}\sup_{f\in \mcF_{s,m}(\bW(\beta,L))}\Ex_f[\|\widehat f-f\|_2^2]\ge C'\min\bigg\{L,\,ms\eps^2\log\bigg(\frac{8\,d}{sm^{1/s}}\bigg)\bigg\}\,.
$$
\end{proposition}
\begin{proof}
We use again Theorem~2.7 in \cite{Tsybakov09} but with a choice of the finite subset of $\mcF_{s,m}(\bW(\beta,L))$ different from that of Proposition~\ref{prop1}.
First, we introduce some additional notation.
For every triplet $(m,s,d)\in\NN_*^3$ satisfying $ms\le d$, let $\mcP_{s,m}^d$ be the set of collections
$\pi=\{V_1,\ldots,V_m\}$ such that each $V_\ell\subseteq\{ 1,\dots,d\}$ has exactly $s$ elements and $V_\ell$'s are pairwise disjoint.
We consider $\mcP_{s,m}^d$ as a metric space with the distance $\rho(\pi,\pi')= \frac1m\sum_{\ell=1}^m \1(V_\ell\not\in\{V'_1,\ldots,V'_m\})= \frac{|\pi \Delta \pi'|}{2m}\,$,  where $\pi'=\{V_1',\ldots,V_m'\}\in\mcP_{s,m}^d$. It is easy to see that $\rho(\cdot,\cdot)$ is a distance bounded by $1$.

For any $\vartheta\in(0,1)$, let $\mcN^d_{s,m}(\vartheta)$ denote the logarithm of the packing number, \textit{i.e.}, the logarithm of the
largest integer $K$ such that there are $K$ elements $\pi^{(1)},\ldots,\pi^{(K)}$ of $\mcP_{s,m}^d$ satisfying
$\rho(\pi^{(k)},\pi^{(k')})\ge \vartheta$. To each $\pi^{(k)}$ we associate a family of
functions ${\mathcal U}=\{f_{k,\oomega}:\oomega\in\{-1,1\}^{ms}, \,k=1,\dots,K\}$ defined by
$$
f_{k,\oomega}(\bx)=\frac{\tau}{\sqrt{m}} \sum_{V\in\pi^{(k)}} \varphi_{\oomega,V}(\bx_V),
$$
where $\tau=(1/4)\min\big(\eps \sqrt{ms\log 2+\log K}, \sqrt{L}\big)$ and $\varphi_{\oomega,V}(\bx_V)=\prod_{j\in V} \varphi_{\omega_j}(x_j)$.
Using that $\{\varphi_{\bj}\}$ is the tensor-product trigonometric basis it is easy to see that each $f_{k,\oomega}$ belongs to $\mcF_{s,m}(\bW(\beta,L))$.
Next, $|{\mathcal U}|=2^{ms}K$ and, for any $f_{k,\oomega}\in{\mathcal U}$, the Kullback-Leibler divergence
between $\Pb_{f_{k,\oomega}}$ and $\PbO$ is equal to $\KL(\Pb_{f_{k,\oomega}},\PbO)=\frac12\eps^{-2}\|f_{k,\oomega}\|_2^2=\frac{\eps^{-2}\tau^2}{2} \le \frac{\log |{\mathcal U}|}{16}$.
Furthermore, the functions  $f_{k,\oomega}$ are not too close to each other. Indeed, since $\{\varphi_{\bj}\}$ is the tensor-product trigonometric basis we get that, for all $f_{k,\oomega},f_{k',\oomega'}\in{\mathcal U}$,
\begin{align*}
\|f_{k,\oomega}-f_{k',\oomega'}\|_2^2
        &=\tau^2m^{-1}\Big(2m-\sum_{V\in \pi^{(k)}}\sum_{V'\in \pi^{(k')}} \int_{[0,1]^d}\varphi_{\oomega,V}(\bx_V)\varphi_{\oomega',V'}(\bx_{V'})\,d\bx\Big)\\
        &=\tau^2\Big(2-\frac1m\sum_{V\in \pi^{(k)}}\sum_{V'\in \pi^{(k')}} \1(V=V')\Big)
        =2\tau^2\rho(\pi^{(k)},\pi^{(k')})\ge 2\vartheta \tau^2.
\end{align*}
These remarks and Theorem~2.7 in \cite{Tsybakov09} imply that
\begin{align}\label{inf-1}
\inf_{\widehat f}\sup_{f\in {\mathcal U}}\Ex_f[\|\widehat f-f\|_2^2]
    &\ge c_3 \vartheta \tau^2 = \frac{c_3\vartheta}{16} \min\Big\{L,\eps^2(ms\log 2+\log K)\Big\}
\end{align}
for some absolute constant $c_3>0$. Assume first that $d<4sm^{1/s}$. Then $ms\log 2\ge  \frac{ms}5\log\big(\frac{8d}{sm^{1/s}}\big)$ and the result of the proposition is straightforward. If $d\ge 4sm^{1/s}$ we fix $\vartheta=1/8$ and use the following lemma (cf.\ the Appendix for a proof)
to bound $\log K=\mcN_{s,m}^d(\vartheta)$ from below.

\begin{lemma}\label{lem:3}
For any $\vartheta\in (0,1/8]$ we have
$
\mcN^d_{s,m}(\vartheta)\ge -m\log\big(\frac{8e^{7/8}s^{1/2}}{7}\big)+\frac{ms}{3}\log\big(\frac{d}{sm^{1/s}}\big) .
$
\end{lemma}

This yields
\begin{align}\label{inf-2}
ms\log 2+\mcN_{s,m}^d(\vartheta)
    &\ge    \frac{ms}3\log\Big(\frac{8d}{sm^{1/s}}\Big)-m\log\big({(8/7)e^{7/8}s^{1/2}}\big).
\end{align}
It is easy to check that
$
m\log\left({(8/7)e^{7/8}s^{1/2}}\right)\le 1.01ms,
$
while for $d\ge 4sm^{1/s}$ we have
$
\frac13\log\Big(\frac{8d}{sm^{1/s}}\Big)\ge 1.15.
$
Combining these inequalities with (\ref{inf-1}) and  (\ref{inf-2}) we get the result.
\end{proof}

\section{Discussion and outlook}

We presented a new framework, called the compound functional model, for performing various statistical
tasks such as prediction, estimation and testing in the context of high dimension. We studied the
problem of estimation in this model from a minimax point of view when the data are generated by
a Gaussian process.  We established upper and lower bounds on the minimax risk that match
up to a multiplicative constant. These bounds are nonasymptotic and are attained adaptively with
respect to the macroscopic and microscopic sparsity parameters $m$ and $s$, as well as to the
complexity of the atoms of the model.  In particular, we improve in several aspects upon the existing
results for the sparse additive model, which is a special case of the compound functional model (only
for this case the rates were previously explicitly treated in the literature):
\vspace{-4mm}
\begin{itemize}
\item The exact expression for the optimal rate that we obtain reveals that the existing methods for the sparse additive model based on penalized least squares techniques have logarithmically suboptimal rates.
\item On the difference from most of the previous work, we do not require restricted isometry type assumptions on the subspaces of the additive model; we need only a much weaker one-sided condition (\ref{3a}). Possible extensions to general compound model based on the existing literature would again suffer from the rate suboptimality and require such type of extra conditions.
\item When specialized to the sparse additive model,
our results are adaptive with respect to the smoothness of the atoms, while all the previous
work about the rates considered the smoothness (or the reproducing kernel) as given in advance.
\end{itemize}
For the general compound model, the main difficulty is in the proof of the lower bounds of the order $ms \varepsilon^2\log( d/(s m^{1/s}))$ that are not covered by the standard tools such as the Varshamov-Gilbert lemma or $k$-selection lemma. Therefore, we developed here new tools for the lower bounds that can be of independent interest.

An important issue that remained out of scope of the present work but is undeniably worth studying is
the possibility of achieving the minimax rates by computationally tractable procedures. Clearly, the
complexity of exact computation of the procedure described in Section~\ref{sec:upper} scales as $\eps^{-2m}2^{M_{d,s}}$,
which is prohibitively large for typical values of $d$, $s$ and $m$. It is possible, however, to approximate
our estimator by using a Markov Chain Monte-Carlo (MCMC) algorithm similar to that of
\cite{RigTsy11,RigTsy12}.
The idea is to begin with an initial state $(\bt_0,\eeta_0)$ and to randomly generate
a new candidate $(\bu,\zzeta)$ according to the distribution $q(\cdot|\bt_0,\eeta_0)$, where $q(\cdot|\cdot)$ is a
given Markov kernel. Then, a Bernoulli random variable $\xi$  with probability of the output 1 equal to
$\alpha= 1\wedge\frac{\widehat\pi(\bu,\zzeta)}{\widehat\pi(\bt,\eeta)}\frac{q(\bt,\eeta|\bu,\zzeta)}{q(\bu,\zzeta|\bt,\eeta)}$
is drawn and a new state $(\bt_1,\eeta_1)=\xi\cdot(\bu,\zzeta)+(1-\xi)\cdot(\bt_0,\eeta_0)$ is defined.
This procedure is repeated $K$ times producing thus a realization $\{(\bt_k,\eeta_k); k=0,\ldots,K\}$ of a
reversible Markov chain. Then, the average value $\frac1K\sum_{k=1}^K \widehat\ttheta_{\bt_k,\eeta_k}$ provides an
approximation to the estimator $\widehat f_\eps$ defined in Section~\ref{sec:upper}.

If $s$ and $m$ are small and $q(\cdot|\bt,\eeta')$ is such that all the mass of this distribution
is concentrated on the nearest neighbors of the $\eeta'$ in the hypercube of $2^{M_{d,s}}$ all possible $\eeta$'s,
then the computations can be performed in a polynomial time. For example, if $s=2$, \textit{i.e.}, if we allow
only pairwise interactions, each step of the algorithm requires $\sim\eps^{-2m}d^2$ computations, where the
factor  $\eps^{-2m}$ can be reduced to a power of $\log(\eps^{-2})$ by a suitable modification of the estimator.
How fast such MCMC algorithms converge to our estimator and what is the
most appealing choice for the Markov kernel $q(\cdot|\cdot)$ are challenging open
questions for future research.

\appendix
\section*{Appendix}

\section{Proof of Lemma~\ref{lem:1}}
Let $\eeta\in\tilde\mcB$  be such that $f\in \mcF_\eeta(\bW(\beta,L))$ and  $\supp(\eeta)=\{V_1,\ldots,V_m\}$ where $|V_\ell|\le s$.
Then there exist a constant $\bar f$ and $m$ functions $f_1,\ldots,f_m$  such that $f_\ell\in W_{V_\ell}(\beta,L)$, $\ell=1,\dots,m$,
and $f=\bar f+f_1+\ldots+f_m$. Set $\theta_{\bj,\ell}=\theta_\bj[f_\ell]$, $(\bj,\ell)\in\ZZ^d\times\{1,\dots,m\}$. Using the notation  $t_\ell=t_{V_\ell}$
and $$\mcJ=\{\bj\in\ZZ^d: \exists \,\ell\in\{1,\dots,m\} \  \text{such that} \ \supp(\bj)\subseteq V_\ell\ \text{and}\ |\bj|_\infty\le t_{\ell} \}$$
we get
\begin{align*}
\widehat f_{\bt,\eeta}-f
    &= (Y_{\bf 0}-\bar f)\varphi_{\bf 0}+\sum_{\bj\in\mcJ\setminus{\bf 0}} \widehat\theta_{\bt,\eeta,\bj}\varphi_\bj-\sum_{\ell=1}^m \sum_{\bj\in\ZZ^d\setminus{\bf 0}} \theta_{\bj,\ell}\varphi_\bj\1_{\{\supp(\bj)\subseteq V_\ell;|\bj|_\infty\le t_{\ell}\}}  \\
    &\quad
            -\sum_{\ell=1}^m\sum_{\bj\in\ZZ^d} \theta_{\bj,\ell}\varphi_\bj\1_{\{\supp(\bj)\subseteq V_\ell;|\bj|_\infty> t_{\ell}\}}\\
    &= \sum_{\bj\in\mcJ} \eps \xi_\bj\varphi_\bj
        -
        \sum_{\ell=1}^m\sum_{\bj\in\ZZ^d} \theta_{\bj,\ell}\varphi_\bj\1_{\{\supp(\bj)\subseteq V_\ell;|\bj|_\infty> t_{\ell}\}},
\end{align*}
where $(\xi_\bj)_{\bj\in\ZZ^d}$ are i.i.d.\ Gaussian random variables with zero mean and variance one. In view of the bias-variance decomposition and
(\ref{3a}), we bound the risk of $\widehat f_{\bt,\eeta}$ as follows:
\begin{align}
\Ex_f[\|\widehat f_{\bt,\eeta}-f\|_2^2]
		&\le \sum_{\bj\in\mcJ} \eps^2  +C_* \sum_{\ell=1}^m\sum_{\bj\in\ZZ^d} \theta_{\bj,\ell}^2\1_{\{\supp(\bj)\subseteq V_\ell;|\bj|_\infty> t_{\ell}\}}
\nonumber\\
		&\le \sum_{\ell=1}^m\sum_{\bj\in\ZZ^d} \eps^2 \1_{\{\supp(\bj)\subseteq V_\ell;|\bj|_\infty\le t_{\ell}\}}
            +C_*\sum_{\ell=1}^m\sum_{\bj\in\ZZ^d} \theta_{\bj,\ell}^2\1_{\{\supp(\bj)\subseteq V_\ell;|\bj|_\infty> t_{\ell}\}}\nonumber\\
		&\le m \max_{\ell=1,\dots,m} \sum_{\bj:\supp(\bj)\subseteq V_\ell} \Big(\eps^2 \1_{\{|\bj|_\infty\le t_{\ell}\}}
            +C_*\theta_{\bj,\ell}^2\1_{\{|\bj|_\infty> t_{\ell}\}}\Big)\label{eq:10}.
\end{align}
In the right-hand side of (\ref{eq:10}), the first summand is the variance term, while the second summand is the (squared) bias term of the risk. We bound these two terms separately.
For the bias contribution to the risk, we find:
\begin{align}
\sum_{\bj:\supp(\bj)\subseteq V_\ell} \theta_{\bj,\ell}^2\1_{\{|\bj|_\infty> t_{\ell}\}}	
	&\le (t_\ell+1)^{-2\beta}\sum_{\bj:\supp(\bj)\subseteq V_\ell} {|\bj|}_\infty^{2\beta}\theta_{\bj,\ell}^2\nonumber\\
	&\le L(t_\ell+1)^{-2\beta}\nonumber\\
	&\le L\bigg(\eps^{4\beta} \vee (L/(3^{|V_\ell|}\eps^2))^{-2\beta/(2\beta+|V_\ell|)}\bigg)\nonumber\\
		&\le 3^{2\beta\wedge s}\big(L\eps^{4\beta} \vee L^{s/(2\beta+s)}\eps^{4\beta/(2\beta+s)}\big).\label{eq:3}
\end{align}
If $t_\ell\ge 1$, then the variance contribution to the risk is bounded as follows:
\begin{align}
\sum_{\bj:\supp(\bj)\subseteq V_\ell}\eps^2\1_{\{{|\bj|}_\infty\le t_\ell\}}
		&=\eps^2(2t_\ell+1)^{|V_\ell|}\le \eps^2(3t_\ell)^{|V_\ell|}
		\le 3^{2\beta\wedge s}L^{|V_\ell|/(2\beta+|V_\ell|)}\eps^{4\beta/(2\beta+|V_\ell|)},\label{eq:4}
\end{align}
where we have used that $t_\ell\le (L/3^{|V_\ell|}\eps^2)^{1/(2\beta+|V_\ell|)}$ and $|V_\ell|\le s$.
Finally, note that condition $\log(\eps^{-2})\ge (2\beta)^{-1}\log(L)$ implies that $L\eps^{4\beta} < L^{s/(2\beta+s)}\eps^{4\beta/(2\beta+s)}$ in (\ref{eq:3}).
Thus, inequality (\ref{eq:10}) together with (\ref{eq:3}) and (\ref{eq:4}) yields the lemma in the case $t_\ell\ge 1$. If $t_\ell<1$, \textit{i.e.}, $t_\ell=0$,
the same arguments imply that the bias is bounded by $L$ and the variance is bounded by $\eps^2$. Since $L\ge \eps^2$, the sum at the right-hand side of (\ref{eq:10}) is
bounded by $(1+C_*)L$. One can check that $t_\ell$ equals $0$ only if $L<3^s\eps^{-2}$, and in this case $L= L^{s/(2\beta+s)}L^{2\beta/(2\beta+s)}
\le L^{s/(2\beta+s)}\eps^{4\beta/(2\beta+s)} 3^{2\beta s/(2\beta+s)}\le 3^{2\beta\wedge s}L^{s/(2\beta+s)}\eps^{-4\beta/(2\beta+s)} $. This completes the proof.

\section{Proof of Lemma~\ref{lem:3}}
Prior to presenting a proof of Lemma~\ref{lem:3}, we need an additional result.
\begin{lemma}\label{lem:2}
For a triplet $(m,s,d)\in\NN_*^3$ satisfying $ms\le d$, let $\mcP_{s,m}^d$ be the set of all collections
$\pi=\{A_1,\ldots,A_m\}$ with $A_i\subseteq\{1,\dots,d\}$ such that $|A_i|=s$ for all $i$ and $A_i\cap A_k=\emptyset$
for $i\ne k$. Then
$$
s^{-(m-1)/2}\bigg(\frac{d}{sm^{1/s}}\bigg)^{ms}\le |\mcP_{s,m}^d|\le \bigg(\frac{e^2d}{sm^{1/s}}\bigg)^{ms}.
$$
\end{lemma}
\begin{proof}
Using standard combinatorial arguments we find
$$
\mcP_{s,m}^d=\binom{d}{ms} \frac{(ms)!}{(s!)^m m!}\ge \bigg(\frac{d}{ms}\bigg)^{ms} \frac{(ms)!}{(s!)^m m!}.
$$
If either $s=1$ or $m=1$ then ${(ms)!}={(s!)^m m!}$ and the lower bound stated in the lemma
is obviously true. Assume now that $m\ge 2$ and $s\ge2$.
Recall that according to the Stirling formula, for every $n\in\NN$,
$\sqrt{2\pi n}(n/e)^{n}\le n!\le \sqrt{2\pi n}(n/e)^{n} e^{1/12n}$. Therefore,
\begin{align*}
\frac{(ms)!}{m! (s!)^m } &\geq \frac{\sqrt{2\pi ms}(ms/e)^{ms}}{\sqrt{2\pi m}(m/e)^{m} e^{\frac1{12m}+\frac{m}{12s}}(\sqrt{2\pi s})^m(s/e)^{ms}}\\
	&=\frac{m^{ms}}{m^m}\Big[e^{1-\frac1{12m^2}-\frac1{12s}}/\sqrt{2\pi}\Big]^m  s^{-(m-1)/2}\,.
\end{align*}
Since the expression in square brackets in the last display is greater than 1 we obtain the desired lower bound on $|\mcP_{s,m}^d|$. The upper bound follows from (\ref{eq:1}) and the fact that $|\mcP_{s,m}^d|\le |\mcB_{s,m}^d|$.
\end{proof}

\textbf{Proof of Lemma~\ref{lem:3}}. Consider first the case $m=1$. The set $\mcP^d_{s,1}$ is the collection of all subsets of $\{ 1,\dots,d\}$ having
exactly $s$ elements. The distance $\rho$ is then 0 if the sets coincide and 1 otherwise. Thus, we need to
bound from below the logarithm of $|\mcP^d_{s,1}|=\binom{d}{s}$. It is enough to use the inequality
$\log \binom{d}{s}\ge s\log(d/s)$.


Assume now that $m\ge 2$.
Since $\pi^{(1)},\ldots,\pi^{(K)}$ is a \textit{maximal} $\vartheta$-separated set of $\mcP_{s,m}^d$ we have that $\mcP_{s,m}^d$ is covered
by the union of $\rho$-balls $B(\pi^{(k)},\vartheta)$ of radius $\vartheta$ centered at $\pi^{(k)}$'s. Therefore,
$$
|\mcP_{s,m}^d|\le \sum_{k=1}^K |B(\pi^{(k)},\vartheta)|.
$$
It is clear that the cardinality of the ball  $|B(\pi^{(k)},\vartheta)|$ does not depend on $\pi^{(k)}$. This yields
$$
K\ge \frac{|\mcP_{s,m}^d|}{|B(\pi^{0},\vartheta)|}
$$
where $\pi^0=\{A^0_1,\ldots,A^0_m\}$ such that $A^0_i=\{ (i-1)s+1,\dots, is\}$.
We have already established a lower bound on $|\mcP_{s,m}^d|$ in Lemma~\ref{lem:2}. We now find an upper bound on the cardinality of the
ball $B(\pi^{0},\vartheta)$. Let $m_\vartheta$ be the smallest integer greater than or equal to $(1-\vartheta)m$.  Consider some $\pi=\{A_1,\ldots,A_m\}\in\mcP_{s,m}^d$. Note that $\pi\in B(\pi^{0},\vartheta)$ if and only if
$$
\sum_{i=1}^m \1(A_i\in\{A^0_1,\ldots,A^0_m\}) \ge m_\vartheta.
$$
This means that there are $m_\vartheta$ indexes $i_1,\ldots,i_{m_\vartheta}$ such that
the $m_\vartheta$ sets $A^0_{i_j}$ are in $\pi$ and the remaining $m-m_\vartheta$ elements of $\pi$ are chosen
as an arbitrary collection of $m-m_\vartheta$ disjoint subsets of $\{ 1,\dots,d\}\setminus \bigcup_{j=1}^{m_\vartheta} A^0_{i_j}$, each of which  is of cardinality  $s$. There are $\binom{m}{m_\vartheta}$ ways of choosing
$\{i_1,\ldots,i_{m_\vartheta}\}$ and once this choice is fixed, there are $|\mcP_{s,m-m_\vartheta}^{d-sm_\vartheta}|$
ways of choosing the remaining parts. Thus,  $|B(\pi^{0},\vartheta)|\le \binom{m}{m_\vartheta}|\mcP_{s,m-m_\vartheta}^{d-sm_\vartheta}|$. Using this inequality and Lemma~\ref{lem:2} we obtain
\begin{align*}
K   &\ge \frac{|\mcP_{s,m}^d|}{\binom{m}{m_\vartheta}|\mcP_{s,m-m_\vartheta}^{d-sm_\vartheta}|}
    \ge \frac{s^{-(m-1)/2}\Big(\frac{d}{sm^{1/s}}\Big)^{ms}}{\Big(\frac{em}{m_\vartheta}\Big)^{m_\vartheta}\Big(\frac{e^2(d-sm_\vartheta)}{sm_\vartheta^{1/s}}\Big)^{s(m-m_\vartheta)}} \\
    &\ge s^{-(m-1)/2}e^{2s(m_\theta-m)-m_\theta}\Big(\frac{d}{sm^{1/s}}\Big)^{sm_\vartheta} \Big(\frac{m_\vartheta}{m}\Big)^{m}
     \bigg(1+\frac{sm_\theta}{d-sm_\theta}\bigg)^{s(m-m_\theta)}.
\end{align*}
Since $\vartheta\le 1/8$ we have $
m_\vartheta\ge m\big(1-\vartheta\big)\ge 7m/8$ and after some algebra we deduce from the previous display that
\begin{align}\label{eq:append}
\log(K)
	&\ge -\frac{ms}{4}-m\log\bigg(\frac{8e^{7/8}s^{1/2}}{7}\bigg)+\frac{7ms}8\log\bigg(\frac{d}{sm^{1/s}}\bigg).
\end{align}
Assume first that $s\ge 3$. Since also $m\ge 2$ we have
$2^{1-1/3}\le m^{1-1/s}=\frac{ms}{sm^{1/s}}\le \frac{d}{sm^{1/s}}$. Hence
\begin{align*}
\log(K)
	&\ge -\frac{ms}{4\log(2^{2/3})}\log\bigg(\frac{d}{sm^{1/s}}\bigg)-m\log\bigg(\frac{8e^{7/8}s^{1/2}}{7}\bigg)+\frac{7ms}8\log\bigg(\frac{d}{sm^{1/s}}\bigg)
\end{align*}
and the result of the lemma follows from the inequality $7/8-1/\log (2^{8/3})\ge 1/3$. It remains to consider the case $s\in\{1,2\}$. If the right-hand side
of the inequality of the lemma is negative, then the result is trivial. If the right-hand side is positive, we have
$\log(8e^{7/8}/{7})\le \frac{2}{3}\log\big(\frac{d}{sm^{1/s}}\big)$ for $s\in\{1,2\}$.  Therefore, from (\ref{eq:append}) we obtain
\begin{align*}
\log(K)
	&\ge -\frac{ms}{6\log(8e^{7/8}/{7})}\log\bigg(\frac{d}{sm^{1/s}}\bigg)-m\log\bigg(\frac{8e^{7/8}s^{1/2}}{7}\bigg)+\frac{7ms}8\log\bigg(\frac{d}{sm^{1/s}}\bigg)
\end{align*}
and the result of the lemma follows from the inequality $7/8-(6\log(8e^{7/8}/{7}))^{-1}\ge 1/2$.

\section*{Acknowledgments}
The authors acknowledge the support of the French Agence Nationale de la Recherche (ANR) under the grant PARCIMONIE.

\bibliography{ref_all}


\bigskip

\end{document}